\documentclass[a4paper,final]{siamart190516}

\usepackage{damacros}

\headers{Neural fields with dendrites }{D. Avitabile, S. Coombes, P.~M. Lima}

\title{Numerical Investigation of a Neural Field Model Including Dendritic Processing} 

\author{%
  Daniele Avitabile%
  \thanks{%
    Vrije Universiteit Amsterdam,
    Department of Mathematics,
    Faculteit der Exacte Wetenschappen,
    De Boelelaan 1081a,
    1081 HV Amsterdam, The Netherlands.
  \protect\\
    Inria Sophia Antipolis M\'editerran\'ee Research Centre,
    MathNeuro Team,
    2004 route des Lucioles-Boîte Postale 93 06902,
    Sophia Antipolis, Cedex, France.
  \protect\\
    (\email{d.avitabile@vu.nl}, \url{www.danieleavitabile.com}).
  }
  \and
  Stephen Coombes \thanks{Centre for Mathematical Medicine and Biology, School of
  Mathematical Sciences, University of Nottingham, NG7 2RD, UK.}
  \and
  Pedro M. Lima \thanks{CEMAT, Instituto Superior Tecnico,University of Lisbon,
  Portugal}
}

\usepackage[ruled,linesnumbered]{algorithm2e}

\newcommand{\Real}{\operatorname{Re}}    
\newcommand{\Imag}{\operatorname{Im}}  

\newcommand{\Cb}{\mathbb{C}}
\newcommand{\Rb}{\mathbb{R}}
\newcommand{\Nb}{\mathbb{N}}
\newcommand{\Zb}{\mathbb{Z}}
\newcommand{\Ib}{\mathbb{I}}
\newcommand{\cF}{\mathcal{F}}

\newcommand{\diff}{\mathop{}\!\mathrm{d}}
\newcommand{\ep}{\varepsilon}

\graphicspath{ {Figures/} }

\usepackage{ctable}
\newcommand{\otoprule}{\midrule[\heavyrulewidth]}
\usepackage{color}
\usepackage{listings}

\definecolor{LightGrey}{rgb}{0.9629411,0.9629411,0.9629411}
\definecolor{LighterGrey}{gray}{0.99}
\definecolor{Mauve}{rgb}{0.58,0,0.82}
\definecolor{Emerald}{rgb}{0.31, 0.78, 0.47}
\definecolor{RoyalBlue}{rgb}{0.25, 0.41, 0.88}
\definecolor{myGreen}{cmyk}{0.82,0.11,1,0.25}

\begin{document}

\maketitle

\begin{abstract}
We consider a simple neural field model in which the state variable is dendritic
voltage, and in which somas form a continuous one-dimensional layer.
This \textit{neural field} model with \textit{dendritic processing} is formulated as
an integro-differential equation. We introduce a
computational method for approximating solutions to this nonlocal model, and use
it to perform numerical simulations for neuro-biologically realistic choices of
anatomical connectivity and nonlinear firing rate function.  For the time
discretisation we adopt an Implicit-Explicit (IMEX) scheme; the space discretisation
is based on a finite-difference scheme to approximate the diffusion term and uses the
trapezoidal rule to approximate integrals describing the nonlocal interactions in
the model. We prove that the scheme is of first-order in time and second order in
space, and can be efficiently implemented if the factorisation of a small, banded
matrix is precomputed. By way of validation we compare the outputs of a numerical
realisation to theoretical predictions for the onset of a Turing pattern, and to the
speed and shape of a travelling front for a specific choice of Heaviside firing rate.
We find that theory and numerical simulations are in excellent agreement.

\end{abstract}

\section{Introduction}

Ever since Hans Berger made the first recording of the human electroencephalogram
(EEG) in 1924 there has been a tremendous interest in understanding the physiological
basis of brain rhythms. This has included the development of mathematical models of
cortical tissue, which are often referred to as neural field models.  
The formulation of these models has not changed much since the seminal work of Wilson
and Cowan, Nunez and Amari in the 1970s, as recently described in \cite{Coombes2014}.
Neural fields and neural mass models approximate neural activity assuming the
cortical tissue is a continuous medium. They are coarse-grained spatiotemporal models,
which lack important physiological mechanisms known to be fundamental in generating brain
rhythms, such as dendritic structure and cortical folding. Nonetheless their basic
structure has been shown to provide a mechanistic starting point for understanding
whole brain dynamics, as described by Nunez \cite{Nunez1995}, and especially that of
the EEG.

Modern biophysical theories assert that EEG signals from a single scalp
electrode arise from the coordinated activity of $\sim 10^8$ pyramidal cells in the
cortex.  These are arranged with their dendrites in parallel and perpendicular to the
cortical surface. When activated by synapses at the proximal
dendrites, extracellular current flows parallel to the dendrites, with a net
membrane current at the synapse. For excitatory (inhibitory) synapses this creates a
sink (source) with a negative (positive) extracellular potential.  Because there is
no accumulation of charge in the tissue the proximal synaptic current is compensated
by other currents flowing in the medium causing a distributed source in the case of a
sink and vice-versa for a synapse that acts as a source. Hence, at the population
level the potential field generated by a synchronously activated population of
cortical pyramidal cells behaves like that of a dipole layer.  Although the important
contribution that single dendritic trees make to generating extracellular electric
field potentials has been known for some time, and can be calculated using
Maxwell's equations \cite{Pettersen08}, they are often not accounted for in neural
field models.  However, with the advent of laminar electrodes to record from
different cortical layers it is now timely to build on early work by Crook and
coworkers \cite{crook1997role} and by Bressloff, reviewed in \cite{Bressloff97}, and develop neural field models that incorporate a
notion of dendritic depth.  This will allow a significant and important departure
from present-day neural field models, and recognise the contribution of dendritic
processing to macroscopic large-scale brain signals. A simple way to generalise
standard neural field models is to consider the dendritic cable model of Rall as the
core component in a neural field, with source terms on the cable mediating
long-range synaptic interactions.  These in turn can be described with the
introduction of an \textit{axo-dendritic} connectivity function.

Here we consider a neural field model which treats the voltage on a dendrite as the
primary variable of interest in a simple model of neural tissue. The model comprises
a continuum of somas (a \emph{somatic layer}, see schematic in
Figure~\ref{fig:sketch}(a)). Dendrites are modeled as unbranched fibres,
orthogonal to the somatic layer which, for simplicity, is one-dimensional and
rectilinear (see Figure~\ref{fig:sketch}(b)). At each point along the somatic layer
$x \in \Rb$ we envisage a fibre with coordinate $\xi \in \Rb$. The voltage dynamics
along the fibre is described by the cable equation, with a nonlocal input current
arising as
an integral
over the outputs from the somatic layer (where $\xi=0$). Denoting the voltage by
$V(x,\xi,t)$ we have an integro-differential equation for the
real-valued function $V:  \Rb^2 \times \Rb  \rightarrow \Rb$ of the form
  \begin{multline} \label{1}
  \partial_t V(x,\xi, t) = (-\gamma + \nu \partial_{\xi \xi}) V(x,\xi,t)
                        + G(x,\xi,t)
			\\
  + \int_{\Rb^2} W(x,\xi,y,\eta) 
                        S(V(y,\eta,t))\diff y \diff \eta ,
  \end{multline}
posed on $(x,\xi,t) \in \Rb^3$, for some typically sigmoidal or
Heaviside-type  firing rate function $S$, and some external input function $G$. Here
$\nu$ is the diffusion coefficient and $1/\gamma$ the membrane time-constant of the
cable. As we shall see below, it will be crucial for our analysis that currents flow
exclusively along the fibres, that is, the diffusive term in \eqref{1} contains
derivatives only with respect to $\xi$.

The model is completed with a choice of the generalised axo-dendritic connectivity
function $W$. The nonlocal input current arises from the somatic layer, hence they
are transferred from sources in an $\ep$-neighbourhood of $\xi = 0$, $0 < \ep \ll 1$, to 
contact points in an $\ep$-neighbourhood of $\xi=\xi_0$ on the cable (see
Figure~\ref{fig:sketch}(b)). In addition, the strength of interaction depends
solely on the distance between the source and the contact point, measured along the
somatic layer, leading to the decomposition
\begin{equation}\label{eq:kernel}
  W(x,\xi,y, \eta)=    w(|x-y|)\delta_\ep(\xi-\xi_0) \delta_\ep(\eta),
\end{equation}
where $w$ describes the strength of interaction across the somatic space and is
chosen to be translationally invariant and $\delta_\ep$ is a quickly-decaying
function. 

\begin{figure}
  \centering
  \includegraphics{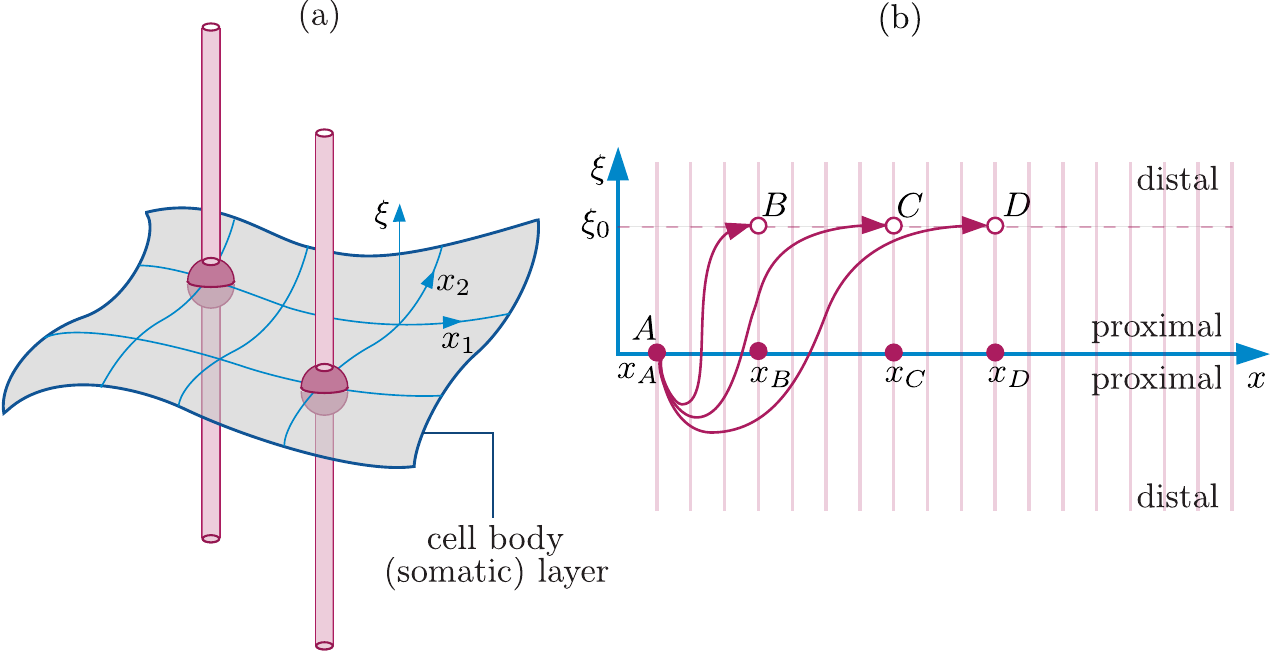}
  \caption{Schematic of the neural field model. (a) Dendrites are represented as
    unbranched fibres (red), orthogonal to a continuum of somas (somatic layer, in
    grey). (b) Model
    \eqref{1} is for a 1D somatic layer, with coordinate
$x \in \Rb$, and fiber coordinate $\xi \in \Rb$. Input currents are generated in a
small neighbourhood of the somatic layer, at $\xi=0$ and are delivered to a contact
point, in a small neighbourhood of $\xi = \xi_0$. The strength of interaction depends
on the distance between sources and contact points, measured along the somatic layer,
hence the inputs that are generated at $A$ and transmitted to $B$, $C$, and $D$ depend
on $|x_B - x_A|$, $|x_C-x_A|$, and $|x_D - x_A|$, respectively (see
\eqref{eq:kernel}).}
  \label{fig:sketch}
\end{figure}

This work introduces a computational method for approximating
solutions to (\ref{1}), subject to suitable initial and boundary conditions,
and applies it to the numerical simulation of the model with kernel given by
(\ref{eq:kernel}). Numerical methods for neural fields in 2-dimensional media have
been introduced recently in flat
geometries~\cite{Rankin2014,hutt2014numerical,lima2015numerical} and on 2-manifolds
embedded in a 3-dimensional space~\cite{Bojak2010,Visser:2017hy}. In addition,
several available open-source codes, such as the Neural Field
Simulator~\cite{Nichols2015}, the Brain Dynamics Toolbox~\cite{Heitmann}, and
NFTsim~\cite{SanzLeon2018}, perform simulations of neural field equations. Numerical
schemes for models of type~\eqref{1} have not been introduced, analysed, or
implemented, and these are the main contributions of the present article.

In Section \ref{Sec:numerical} we describe, analyse, and discuss implementation
details of the numerical method. In Sections \ref{sec:TWTest}--\ref{sec:TuringTest}
we illustrate the performance of the method by means of some numerical experiments,
including problems whose exact solution has known properties. The numerical results
are discussed and their physical meaning is explained. We finish with some
conclusions and discussion in Section \ref{sec:conclusions}.

\section{Numerical Scheme}
\label{Sec:numerical}
Numerical simulations are performed on \eqref{1}, posed on a
bounded, cylindrical somato-dendritic  domain 
\[
  \Omega = \Rb/2L_x\Zb \times (-L_\xi,L_\xi), 
\]
and subject to initial and boundary conditions,
\begin{equation}\label{eq:systemNum}
  \begin{aligned}
    & \partial_t V  = (-\gamma + \nu \partial_{\xi \xi}) V + K(V) + G
    & & \textrm{on $ \Omega \times (0,T]$}, \\
  &  V(\blank,\blank,0)  = V_0
    & & \textrm{on $ \Omega$}, \\
  &  \partial_\xi V(\blank,-L_\xi,\blank) = \partial_\xi V(\blank,L_\xi,\blank) = 0
    & & \textrm{on $ (-L_x,L_x] \times [0,T]$}, \\
  \end{aligned}
\end{equation}
for some positive constants $T$, $L_x$, $L_\xi$. This setup implies
$2L_x$-periodicity in the somatic
direction, and Neumann boundary conditions in the dendritic direction.
We denote by $K$ the integral operator defined by
\[
  (K(V))(x,\xi,t) = \int_{\Omega} W_{\Omega}(x,\xi,y,\eta)
    S(V(y,\eta,t)) \diff y \diff \eta,
  \qquad (x,\xi) \in \Omega.
\]
where $W_{\Omega}$ is the restriction of $W$ on $\Omega$. This restriction implies
that the function $w$ in \eqref{eq:kernel} be substituted by
its periodic extension on $[-L_x,L_x)$. In the remainder of this paper we will omit the
subscript $\Omega$ from $W$, and assume $w$ to be $2L_x$-periodic.

To expose our scheme we introduce
a spatiotemporal discretisation 
using the evenly spaced grid 
$\{(x_j,\xi_i,t_n)\}$ defined by
\[
  \begin{aligned}
    & x_j = -L_x + j h_x, & & j \in \Nb_{n_x}, && h_x = 2L_x/n_x, \\
    & \xi_i   = -L_\xi + (i-1) h_\xi, & & i \in \Nb_{n_\xi}, && h_\xi =
    2L_\xi/(n_\xi-1), \\
    & t_n = n \tau, & & n \in \Nb_{n_t}, && \tau = T/n_t, \\
  \end{aligned}
\]
where we posed $\Nb_k = \{1,2,\ldots,k\}$ for $k \in \Nb$. The scheme we propose uses
the method of lines for \eqref{eq:systemNum}, in conjunction with differentiation
matrices for the diffusive term and a quadrature scheme for the integral operator.

%
Collocating \eqref{eq:systemNum} at the somato-dendritic nodes we obtain
\begin{equation}\label{eq:collocation}
  \begin{split}
    \partial_t V(x_j,\xi_i,t) = (-\gamma + \nu \partial_{\xi\xi}) V(x_j,\xi_i,t) 
   & + K(V)(x_j, \xi_i,t) \\
   & + G(x_j,\xi_i,t) \quad
  (j,i) \in \Nb_{n_x} \times \Nb_{n_\xi},
  \end{split}
\end{equation}
where, with a slight abuse of notation, we denote by $V$ an interpolant to the
function $V$ in \eqref{eq:systemNum} through $\{ (x_j,\xi_i) \}$.
To obtain a numerical solution of the problem we
must choose: (i) an approximation for the linear operator $(-\gamma + \nu
\partial_{\xi\xi})$ at the somato-dendritic nodes; (ii) an approximation for the
integral operator at the same nodes; (iii) a scheme to time step the derived set of
ODEs. 

In the presentation of the scheme, we shall use two equivalent representations for
the voltage approximation: a matricial description 
\begin{equation}\label{eq:VMatrix}
  V(t) = \{ V_{ij}(t) 
  \colon (i,j) \in \Nb_{n_\xi} \times \Nb_{n_x} \} 
    \in \Rb^{n_\xi \times n_x}, \qquad V_{ij}(t) \approx
  V(x_j,\xi_i,t), 
\end{equation}
and a lexicographic vectorial representation, obtained by introducing the
index mapping $k(i,j) = n_\xi(i-1) + j$,
\begin{equation}\label{eq:VVector}
  U(t) = \{ U_{k(i,j)}(t) \colon (i,j) \in \Nb_{n_\xi} \times \Nb_{n_x} \} 
  \in \Rb^{n_x n_\xi}.
\end{equation}
In the latter, we will sometimes suppress the dependence of $k$ on $(i,j)$, for
notational convenience.

\subsection{Discretisation of the linear operator} A simple choice for discretising
the linear differential operator $(-\gamma + \nu \partial_{\xi \xi})$ is to adopt
differentiation matrices~\cite{trefethen2000}. If a differentiation matrix
$D_{\xi \xi} \in \Rb^{n_\xi \times n_\xi}$ is chosen to approximate the action of
the Laplacian operator $\partial_{\xi \xi}$ on twice differentiable, univariate
functions defined on $[-L_\xi,L_\xi]$, satisfying Neumann boundary conditions, and
sampled at nodes $\{ \xi_i \}$, then the action of the operator $-\gamma + \nu
\partial_{\xi\xi}$ on bivariate functions defined on $[-L_x,L_x) \times
[-L_\xi,L_\xi]$, twice differentiable in $\xi$ with Neumann boundary conditions,
sampled at the nodes $\{ (x_j,\xi_i) \}$ with lexicographical ordering $k(i,j)$ is
approximated by the following block-diagonal matrix
\begin{equation*}\label{eq:LinOp}
  -\gamma I_{n_x n_\xi} + \nu I_{n_x} \otimes D_{\xi \xi} = 
  \begin{bmatrix}
    -\gamma + \nu D_{\xi \xi} &                          &        &               \\
                             &-\gamma + \nu D_{\xi \xi} &        &               \\
                             &                          & \ddots &               \\
                             &                          &        &-\gamma + \nu D_{\xi \xi}     \end{bmatrix}
  ,
\end{equation*}
where $I_n$, $n \in \Nb$, is the $n$-by-$n$ identity matrix, and $\otimes$ is the
Kronecker product between matrices. Since the model prescribes diffusion only along
the dendritic coordinate, the
corresponding matrix has a block-diagonal structure \emph{with identical blocks},
which can be exploited to improve performance in numerical computations. The sparsity pattern of a block is dictated by the
underlying scheme to approximate the univariate Laplacian: we have full blocks if
$D_{\xi\xi}$ is derived from spectral schemes, and sparse blocks
for finite-difference schemes. 

\subsection{Discretisation of the nonlinear integral operator} The starting point to
discretise the integral operator is an $m$th order quadrature formula with $q_m$
nodes $\{ (y_l,\eta_l) \colon l \in \Nb_{q_m} \}$ and weights $\{ \sigma_l \colon l
\in \Nb_{q_m} \}$ for the integral of a bivariate function over $\Omega$,
\[
Q(v)=\int_{\Omega} v(y,\eta) \, \diff y  \diff \eta \approx
\sum_{l \in \Nb_m} v(y_l,\eta_l) \sigma_l = Q_m(v).
\]
Using this formula we approximate the nonlinear operator in \eqref{eq:collocation} by
\[
  Q_m(K(V))(x_j,\xi_i,t) = \sum_{l \in \Nb{q_m}} W(x_j,\xi_i,y_l,\eta_l)
  S(V(y_l,\eta_l,t)) \sigma_l .
\]

We stress that, in general, the quadrature nodes $\{ (y_l,\eta_l) \}$ and the
collocation nodes $\{ (x_{k(i,j)}, \xi_{k(i,j)}) \}$ are disjoint. The former are
chosen so as to approximate accurately the integral term, the latter to approximate
the differential operator. When the two grids are disjoint, an interpolation of $V$ with
nodes $\{ (y_l,\eta_l) \}$ is necessary to derive a set of ODEs at the collocation
nodes. In the remainder of this paper we will assume that collocation and quadrature
nodes coincide, so that we can omit the interpolant, for simplicity.

\subsection{Matrix ODE formulation}
Combining the differentiation matrix, the quadrature rule, and the lexicographic
representation \eqref{eq:VVector}
we obtain a set of $n_x n_\xi$ ODEs
\begin{equation}\label{eq:ODEs}
  \begin{aligned}
    \dot U(t) & = (-\gamma I_{n_x n_\xi} + \nu I_{n_x} \otimes D_{\xi \xi} ) U(t) +
  F(U(t),t), \\
  U(0) & = U_0.
  \end{aligned}
\end{equation}
The structure of the differentiation matrix in section \eqref{eq:LinOp}, however,
suggests a rewriting of \eqref{eq:ODEs} in terms of the blocks of the linear
operator, which correspond to ``slices" at constant values of $x$: we recall the
matrix representation \eqref{eq:VMatrix} and obtain an equivalent
matrix ODE formulation
\begin{equation}\label{eq:MatrixODE}
  \dot V(t) = (-\gamma I_{n_\xi} + \nu D_{\xi \xi}) V(t) + N(V(t)) + G(t),
\end{equation}
where $N$ is the matrix-valued function with components $N_{ij}(V) =
Q_m(V)(x_j,\xi_i)$ and $G$ is the matrix with components $G(x_j,\xi_i,t)$. In passing,
we note that the linear part of the equation involves
a multiplication between an $n_\xi$-by-$n_\xi$ matrix and the $n_\xi$-by-$n_x$ matrix $V$. 

\subsection{Time-stepping scheme} The proposed time-stepping scheme for
\eqref{eq:systemNum} is obtained from \eqref{eq:MatrixODE} with the following
choices: (i) a first-order, implicit-explicit (IMEX) time-stepping
scheme~\cite{ascher1995}; (ii) a second-order, centered finite-difference
scheme for the differentiation matrix $D_{\xi \xi}$; (ii) a second-order trapezium
rule for the quadrature rule $Q_m$. As we shall see, these choices bring
a few computational advantages, which will be outlined below. 

IMEX schemes treat the linear (diffusive) part of the ODE implicitly, and the
nonlinear part explicitly, so that the stiff diffusive term is integrated implicitly
to avoid excessively small time steps. The simplest IMEX method uses backward Euler
for the diffusive term, leading to
\begin{equation}\label{eq:IMEX}
  \begin{aligned}
       V^0 & = V_0, \\
       A V^n & = V^{n-1} + \tau N(V^{n-1}) + \tau G^{n-1}, \qquad n \in \Nb,
  \end{aligned}
\end{equation}
where $V^n \approx V(t_n)$, $G^n = G(t_n)$, and $A$ is the matrix
\begin{equation}\label{eq:AMatr}
  A = (1+\gamma \tau) I_{n_\xi} - \tau \nu D_{\xi \xi}.
\end{equation}
In concrete calculations we use second-order centred finite differences, leading to
\begin{equation}\label{eq:FinDiffLapl}
  D_{\xi \xi} = 
   \Delta/h_\xi^{2}, \qquad \Delta = 
  \begin{bmatrix}
    -2 &  2     &         &        &    &     \\
     1 & -2     &       1 &        &    &     \\
       & \ddots & \ddots & \ddots  &    &     \\
       &        &      1 &      -2 & 1  &     \\
       &        &        &       2 & -2 &  
   \end{bmatrix},
 \end{equation}
in which Neumann boundary conditions are included in the differentiation matrix. 

Finally, we discuss the choice of the quadrature scheme. We use a composite trapezium
scheme with nodes $\{x_j\}$ and weights $\{\rho_j\}$ in $x$, and nodes $\{\xi_i\}$
and weights $\{ \sigma_i\}$ in $\xi$, respectively, hence quadrature and collocation
sets coincide,
\begin{equation} \label{eq:NQuad}
  N_{ij}(V) = \sum_{j' \in \Nb{n_x}} \sum_{i' \in \Nb{n_\xi}}
  W(x_j,\xi_i,x_{j'},\xi_{i'}) S(V_{i',j'}) \rho_{j'}\sigma_{i'}.
  \quad (i,j) \in \Nb_{n_\xi} \times \Nb_{n_x}.
\end{equation}

\subsection{Properties of the IMEX scheme}
In this section we collect some analytical results on the IMEX scheme
\eqref{eq:IMEX}--\eqref{eq:NQuad}. We work with spaces of sufficiently regular
continuous functions, which provides the simplest setting for our results. We denote
by $C^k(D)$ the space of $k$-times continuously differentiable functions from $D$ to
$\Rb$, where $k$ is an integer, $D$ a domain in $\Rb^3$. We also indicate
by $C_b^k(D)$ the space of continuous functions from $D$ to $\Rb$ with bounded and
continuous partial derivatives up to order $k$. 
Both spaces are endowed with the infinity norm $\Vert \blank
\Vert_\infty$. We will use the symbol $| \blank |_\infty$ for the standard
infinity-norm on matrices, induced by the corresponding vector norm. In addition, we
will denote by $\bar D$ the closure of $D$.

We begin with a generic assumption of boundedness on the functions in
\eqref{eq:systemNum}:
\begin{hypothesis}\label{hyp:boundedness}
  There exist $C_W, C_S, C_G >0 $ such that
  \[
    |W| \leq C_W \; \textrm{in $\Omega \times \Omega$},
    \qquad
    |S| \leq C_S \; \textrm{in $\Rb$},
    \qquad
    |G| \leq C_G \; \textrm{in $\Omega \times \Rb$}.
  \]
\end{hypothesis}

\begin{lemma}[Boundedness of IMEX solution]\label{lem:IMEXboundedness} 
  Assume Hypothesis \ref{hyp:boundedness}, then there exists a unique bounded
  sequence $(V^n)_{n\in\Nb}$ satisfying
  the IMEX scheme \eqref{eq:IMEX}--\eqref{eq:NQuad}. In
  addition, the following bound holds
  \[
    \vert V^n \vert_\infty \leq \vert V^0 \vert_\infty + 
    n_x \frac{ \mu(\bar\Omega) C_W C_S + C_G}{\gamma},
    \qquad n \in \Nb.
  \]
\end{lemma}
\begin{proof}
  The matrix $A$ in \eqref{eq:AMatr} has real, strictly positive eigenvalues given
  by
  \[
    \lambda_k = 1 + \gamma \tau + \frac{4\nu \tau}{h_\xi^2} 
    \bigg[ \sin \bigg( \frac{\pi(k-1)}{2n_\xi} \bigg) \bigg]^2, \qquad k \in
    \Nb_{n_\xi},
  \]
  where we have used the fact that the eigenvalues of $D_{\xi \xi}$ are known in closed form.
  We conclude that $A$ is invertible, hence for any fixed $n \in \Nb$,
  the matrix $V^n$ solving \eqref{eq:IMEX} is unique. In addition, $A$ is
  strictly diagonally dominant, hence the following bound holds~\cite{varah75}
  \begin{equation}\label{eq:invABound}
  \vert A^{-1} \vert_\infty \leq \max_{i \in \Nb_\xi} \frac{1}{|A_{ii}| - \sum_{j
  \neq i} |A_{ij}|} = \frac{1}{1+ \gamma \tau}.
  \end{equation}
  To prove boundedness of the sequence $(V^n)_{n \in \Nb}$ we first bound the
  matrices $N(V^{n-1})$, $G^n$ appearing in \eqref{eq:IMEX}
  \[
    \begin{aligned}
    \vert N(V^{n-1}) \vert_\infty 
    & = \max_{i \in \Nb_{n_\xi}} \sum_{j \in \Nb_{n_x}} |N_{ij}(V^{n-1})| \\
    & \leq
    \max_{i \in \Nb_{n_\xi}} \sum_{j \in \Nb_{n_x}} 
                             \sum_{j' \in \Nb{n_x}} \sum_{i' \in \Nb{n_\xi}}
                              |W(x_j,\xi_i,x_{j'},\xi_{i'}) S(V_{i',j'}) \rho_{j'}\sigma_{i'}| \\
    & \leq C_W C_S
    \max_{i \in \Nb_{n_\xi}} \sum_{j \in \Nb_{n_x}} 
                   \sum_{j' \in \Nb{n_x}} \sum_{i' \in \Nb{n_\xi}}
		  \rho_{j'}\sigma_{i'} \\
    & \leq C_W C_S
    \max_{i \in \Nb_{n_\xi}} \sum_{j \in \Nb_{n_x}} \mu(\bar\Omega)
    = n_x \mu(\bar\Omega) C_W C_S,
    \end{aligned},
  \]
  and similarly $\vert G^{n-1} \vert_\infty \leq n_x C_G$, and then combine
  them with the bound for $\vert A^{-1}\vert_\infty$ to find
  \[
    \begin{aligned}
    \vert V^n \vert_\infty 
        & \leq \vert A^{-1} \vert_\infty 
	  \Big(
	      \vert V^{n-1} \vert_\infty 
	    + \tau \vert N(V^{n-1}) \vert_\infty
	    + \tau \vert G^{n-1} \vert_\infty 
	  \Big) \\
	& \leq \frac{1}{1+\gamma \tau}
	  \Big(
	      \vert V^{n-1} \vert_\infty 
	      + \tau n_x \mu(\bar\Omega) C_W C_S 
	      + \tau n_x C_G
	  \Big).
    \end{aligned}
  \]
  We set
  \[
    r = \frac{1}{1+\gamma \tau} <1, \qquad 
    q = \frac{\tau n_x}{1 + \gamma \tau}
      ( \mu(\bar\Omega) C_W C_S + C_G),
  \]
  and use induction and elementary properties of the geometric series to obtain
  \[
    \vert V^n \vert_\infty \leq r^n \vert V^0 \vert_\infty + q \sum_{j =0}^{n-1} r^j
      \leq \vert V^0 \vert_\infty + \frac{q}{1-r},
  \]
  which proves the assertion.
\end{proof}

In addition to proving boundedness of the solution, we address the
convergence rate of the IMEX scheme. For this result, we assume the existence of a
sufficiently regular solution to \eqref{eq:systemNum}.

\begin{lemma}[Local convergence rate of the IMEX scheme]\label{lem:IMEXconvergence}
  Assume Hypothesis \ref{hyp:boundedness}, $W \in C^2(\Omega \times \Omega)$, $S
  \in C^2_b(\Omega)$, and assume \eqref{eq:systemNum}
  admits a strong solution $V_*$ 
  whose partial derivatives $\partial_{tt}V_*$, $\partial_{xx}V_*$,
  $\partial_{x\xi}V_*$, $\partial_{\xi\xi}V_*$, $\partial_{\xi\xi\xi\xi}V_*$
  exist and are bounded on $\bar \Omega \times [0,T]$. Denote 
  by $V^n_*$ the matrix with elements $(V_*^n)_{ij} =
  V_*(x_j,\xi_i,t_n)$, for $(i,j,n) \in \Nb_{n_x} \times \Nb_{n_\xi} \times \Nb_{n_t}$. 
  Further, let $(V^n)_{n \in \Nb}$ be the solution to the IMEX scheme
  \eqref{eq:IMEX}--\eqref{eq:NQuad}, and let
  \[
    \zeta = n_x \mu(\bar\Omega) \Vert W \Vert_\infty \Vert S' \Vert_\infty, 
  \qquad
  h = \max(h_\xi, h_x).
  \]
  There exist constants $C_\tau, C_h >0$ such that
  \begin{align}
     & |V^n - V_*^n|_\infty \leq \frac{1}{\gamma -\zeta}(C_\tau \tau + C_h h^2) 
	  & \text{if $\zeta < \gamma$} \label{eq:bound1}, \\
     & |V^n - V_*^n|_\infty \leq \frac{T}{1+ \gamma \tau} (C_\tau \tau + C_h h^2)
	  & \text{if $\zeta = \gamma$} \label{eq:bound2}, \\
     & |V^n - V_*^n|_\infty \leq 
     \frac{C_\tau \tau + C_h h^2}{\zeta - \gamma} \exp 
     \frac{(\zeta - \gamma) T}{1+\gamma \tau}
          & \text{if $\zeta > \gamma$} \label{eq:bound3}.
  \end{align}
\end{lemma}
\begin{proof}
  The regularity assumptions on $V_*$, and standard results on finite-difference
  approximation and trapezium quadrature rule guarantee the existence of constants
  $C_{tt},C_{xx},C_{\xi\xi},C_{\xi\xi\xi\xi} > 0$ such that for all $n \in \{0 \}
  \cup \Nb_{n_t}$
  \begin{equation}\label{eq:IMEXExact}
    AV_*^n = V_*^{n-1} + \tau \big( N(V_*^{n-1}) + G^{n-1} +  
	C_{tt} \tau + C_{\xi\xi\xi\xi}h^2_\xi + C_{xx}h_x^2 + C_{\xi\xi}h_\xi^2
    \big),
  \end{equation}
  where the errors for the forward finite-difference in $t$, centred finite-difference
  in $\xi$, and trapezium rule are listed progressively, with constants proportional
  to the respective partial derivatives. We subtract \eqref{eq:IMEXExact} from
  \[
    AV^n = V^{n-1} + \tau N(V^{n-1}) + \tau G^{n-1}, 
  \]
  and obtain the error bound
  \begin{equation}\label{eq:intermediateBound}
    |V - V_*^n|_\infty \leq |A^{-1}|_\infty \big(|V - V_*^n|_\infty + \tau |N(V^n) -
    N(V_*^n)|_\infty + \tau \omega \big),
  \end{equation}
  where $\omega = C_\tau \tau + C_h h^2$, $C_\tau = C_{tt}$, $C_h =
  \max(C_{xx},C_{\xi\xi},C_{\xi\xi\xi\xi})$, and
  $h = \max(h_\xi,h_x)$. 
  Since the first derivative $S'$ of $S$ is bounded, we have the following estimate
  for the nonlinear term
  \[
    \begin{aligned}
    |N(V^{n}) - N(V_*^n)|_\infty  
    & \leq
    \Vert W \Vert_\infty
    \Vert S' \Vert_\infty
    \max_{i \in \Nb_{n_\xi}} \sum_{j \in \Nb_{n_x}} 
	       \sum_{j' \in \Nb{n_x}} \sum_{i' \in \Nb{n_\xi}}
	       |V^n_{i'j'}-(V_*^n)_{i'j'}|\rho_{j'}\sigma_{i'} \\
    & \leq
    n_x \mu(\bar \Omega)
    \Vert W \Vert_\infty
    \Vert S' \Vert_\infty
    \vert V^n - V_*^n \vert_\infty \\
    &
    = \zeta \vert V^n - V_*^n \vert_\infty ,
    \end{aligned}
  \]
  which, together with \eqref{eq:invABound} and \eqref{eq:intermediateBound} gives a
  recursive bound for the $\infty$-norm matrix error $|V^n-V_*^n|_\infty$\footnote{The scalar values $r,q$ defined in this proof are
  different from the ones defined in the proof of \cref{lem:IMEXboundedness}.},

  \[
    E^0 = 0, \qquad
    E^n \leq \frac{1+\zeta \tau}{1+\gamma \tau}E^{n-1} + \frac{\tau \omega}{1 + \gamma
    \tau} := r E^{n-1} + q.
    \qquad n \in \Nb_{n_t}.
  \]
  Hence,
  \begin{equation}\label{eq:intermediateBound2}
    E^n \leq q \frac{r^n - 1}{r-1}, \quad r \neq 1,
    \qquad E^n \leq n q \quad r = 1,
    \qquad
    n \in \Nb_{n_t}.
  \end{equation}
  If $ \zeta < \gamma$, then $r < 1$, and we obtain \eqref{eq:bound1} as
  \[
    E^n \leq 
   \frac{q}{1-r} = \frac{\omega}{\gamma - \zeta} = \frac{1}{\gamma -\zeta}(C_\tau
    \tau + C_h h^2),
    \qquad n \in \Nb_{n_t}.
  \]
  If $\zeta = \gamma$, then $r = 1$ and \eqref{eq:bound2} is found as follows
  \[
    E^n \leq n q 
    \leq \frac{n_t\tau \omega}{1 + \gamma\tau}
    = \frac{T}{1 + \gamma\tau}(C_\tau \tau + C_h h^2),
    \qquad n \in \Nb_{n_t}.
  \]
  If $\zeta > \gamma$, then $r > 1$ and we can bound the $n$th term of the sequence
  with an exponential, using the bound $(1+x/n)^n \leq \e^x$ for all $x \in \Rb$, as
  follows,
  \[
    r^n = 
    \bigg(
      1 + \frac{(\zeta - \gamma) n \tau}{n(1+\gamma \tau)}
    \bigg)^n
    \leq
    \exp \frac{(\zeta - \gamma) n \tau}{1+\gamma \tau}
    \leq
    \exp \frac{(\zeta - \gamma) T}{1+\gamma \tau},
  \]
  which combined with \eqref{eq:intermediateBound2} gives \eqref{eq:bound3}:
  \[
    E^n \leq \frac{\omega}{\zeta - \gamma} \exp \frac{(\zeta - \gamma) T}{1+\gamma \tau}
     = \frac{C_\tau \tau + C_h h^2}{\zeta - \gamma} \exp \frac{(\zeta - \gamma)
     T}{1+\gamma \tau}.
  \]
\end{proof}

The preceding lemma shows that the IMEX scheme has first order convergence in time,
and second order convergence in space. As expected, this conclusion holds without
imposing any restriction to the size of $\tau$ in relation to $h$, as happens, for
example, in the case of explicit methods for parabolic equations. In passing we note
that if $\zeta < \gamma$ and $V_*(t)$ exists for all $t \in \Rb$, the error estimate
\eqref{eq:bound1} holds for $n \in \Nb$, that is, in an unbounded interval of time;
on the other hand, the error estimates do not hold on an unbounded time interval when
$\zeta \geq \gamma$, as the bounds depend on $T$.

\subsection{Implementational aspects and efficiency}

In this section we make a few considerations on the implementation of the proposed
IMEX scheme, with the view of comparing its efficiency to an ordinary IMEX scheme,
that is, to an IMEX scheme applied to \eqref{eq:ODEs}.

\subsubsection{Implementation}
IMEX schemes for planar semilinear problems require the inversion of a
discretised Laplacian, which usually is a square matrix with the same dimension of
the problem ($n_\xi n_x$ equations in our case). The particular structure of the
problem under consideration, however, implies that the matrix to be inverted is much
smaller (the square matrix $A$ has only $n_\xi$ rows and $n_\xi$ columns). At each
time step
\eqref{eq:IMEX} we solve a problem of the type $AX=B$, where $A \in \Rb^{n_\xi \times
n_\xi}$, and $X, B \in \Rb^{n_\xi \times n_x}$. This can be achieved efficiently by
pre-computing a factorisation of $A$, and then back-substituting for all
columns of $B$. Since the matrix $A$ is sparse and with low bandwidth, efficient
implementations of the $LU$ decompositions and backsubstitution can be used to solve
the $n_x$ linear problems corresponding to the columns of $X$ and $B$.

An important aspect of the numerical implementation is the evaluation of the
nonlinear term \eqref{eq:NQuad}: evaluating the right-hand side of
\eqref{eq:IMEX} requires in general $O(n^2_\xi n^2_x)$ operations, which is a
bottleneck for the time stepper, in particular for large domains. However, the
structure of the problem can be exploited once again to evaluate this term
efficiently. We make use of the following properties: 
\begin{enumerate}
  \item The kernel $W$ specified in \eqref{eq:kernel} has a product structure,
    hence
    \[
      W(x_j,\xi_i,x_{j'},\xi_{i'}) = \alpha_i \alpha'_{i'}
      w(|x_j-x_{j'}|). 
    \]
    where $\alpha_i = \delta_\ep(\xi_i-\xi_0)$, $\alpha'_{i'} = \delta_\ep(\xi_{i'})$.
    In addition, $w$ is periodic, therefore the matrix with entries $w(|x_j -
    x_{j'}|)$ is circulant with (rotating) row vector $w 
    = \{w(|x_j|) \colon j \in \Nb_{n_x}\}  \in \Rb^{1 \times n_x}$.
   \item The function $x \mapsto V(x,\blank)$ is $2L_x$-periodic, hence the trapezium
     rule has identical weights $\rho_j = h_x$, and the integration
     in $x$ is a circular convolution, which can be performed efficiently in $O(n_x
     \log n_x)$ operations, using the Discrete Fourier Transform (DFT).
\end{enumerate}
We have
\begin{equation}\label{eq:NSlow}
  N_{ij}(V) = \alpha_i \sum_{j' \in \Nb_{n_x}} w_{j-j'} \rho_{j'}
  \sum_{i' \in \Nb_{n_\xi}} \alpha'_{i'} \sigma_{i'} S(V_{i'j'})
  \qquad
  (i,j) \in \Nb_{n_\xi} \times \Nb_{n_x},
\end{equation}
and a DFT can be used to perform the outer
sums~\cite{coombes2012interface,Rankin2014}.
Introducing the direct, $\cF_n$, and inverse, $\cF_n^{-1}$, DFTs for $n$-vectors, we express compactly the nonzero elements of $N$ as
follows
\begin{equation}\label{eq:NFast}
  N = \alpha h_x \cF^{-1}_{n_x} 
   \big[
     \cF_{n_x}[w] \odot  \cF_{n_x}[ (\alpha' \odot \sigma)^T S(V)]
   \big],
\end{equation}
where $\alpha, \alpha', \sigma \in \Rb^{n_\xi \times 1}$
are column vectors, and $\odot$ denotes the Hadamard product, that is, elementwise
vector multiplication. The formula above evaluates the
nonlinear term $N$ in just 
$O(n_xn_\xi) + O(n_x \log n_x)$ operations.

We summarise our implementation with the pseudocode provided in Algorithm
\ref{alg:smart}, and we will henceforth compare quantitatively its efficiency with
a standard IMEX implementation, which we also provide in Algorithm \ref{alg:naive}. The
matricial version, Algorithm \ref{alg:smart} exposes row- and column-vectors, for
which a very compact Matlab implementation can be derived. We give an example of such
implementation in Appendix~\ref{sec:matlab}, and we refer the reader to
\cite{daniele_avitabile_2020_3731920} for a repository of codes used in this article.

\subsubsection{Efficiency estimates}

We now make a few considerations about the efficiency of our algorithm. We will provide
two main measures of efficiency: an estimate of the floating point operations (flops),
and an estimate of the storage space (in floating point numbers) required by the
algorithm, as a function of the input data which, in our case, are the number of
gridpoints in each direction, $n_x$ and $n_\xi$. We are  interested
in how the estimates scale for large $n_x, n_\xi$.

To estimate the number of flops, we count the number of operations required by
Algorithms~\ref{alg:smart} and \ref{alg:naive} in the initialisation step (lines 2--6), and
in a single time step (lines 8--12). We base our estimates on the following
facts and hypotheses:
\begin{enumerate}
  \item The cost of multiplying an $m$-by-$n$ matrix by an $n$-vector is $2mn - m$
    flops.
  \item If an $n$-by-$n$ matrix is tridiagonal, then the matrices $L$ and $U$ of its
    $LU$-factorisation are bidiagonal, and $L$ has $1$ along its main diagonal.
    This implies that storing the $LU$ factorisation requires only $3$
    $n$-vectors. Calculating the $LU$ factorisation costs $2n + 1$ flops, while
    solving the corresponding linear problem $LUx = b$, with $x,b \in \Rb^n$,
    requires $2n-2$ and $3n-2$ flops for the forward- and backward-subsitution,
    respectively. Similar considerations apply if $A$ is not tridiagonal, but still
    sparse, as it would be obtained using a different discretisation method for the
    diffusive operator: estimates for the flops of the corresponding
    $PLU$-factorisation depend, in general, on the sparsity pattern of $A$, as well
    as on the permutation strategy, which is heuristic but can have an impact on the
    sparsity of $L$ and $U$, thereby influencing the performance of the algorithm.
    We present calculations only in the case of a tridiagonal matrix $A$, for which
    explicit estimates are possible.

  \item As stated above, it is well known that a single FFT of an $n$-vector costs
    $O(n \log n)$ operations.

  \item We assume that function evaluations of the functions $G$, $S$, $w$, $\delta$
    cost one flop. This estimate is optimistic, as most function evaluations will
    require more than one flop, but we make this simplifying assumption for both the
    algorithms we are comparing.
\end{enumerate}

\begin{algorithm}
  \caption{IMEX time stepper in matrix form \eqref{eq:MatrixODE}, nonlinear term
  computed with pseudospectral evaluation \eqref{eq:NFast}}
  \label{alg:smart}
\DontPrintSemicolon
\SetAlgoNoLine
\SetKwInOut{Input}{Input}
\SetKwInOut{Output}{Output}

\Input{Initial condition $V^0 \in \Rb^{n_\xi \times n_x}$, time step $\tau$, number of steps $n_t$.}
\Output{An approximate solution $(V^n)_{n=1}^{n_t} \subset \Rb^{n_\xi \times n_x}$}
\Begin{
  Compute grid vectors $\xi \in \Rb^{n_\xi \times 1}$, $x \in \Rb^{1 \times n_x}$. \;
  Compute synaptic vectors $w, \hat w = \mathcal{F}_{n_x}[w] \in \Rb^{1 \times n_x}$.\;
  Compute synaptic vectors $\alpha, \alpha' \in \Rb^{n_\xi \times 1}$. \;
  Compute quadrature weights $\sigma \in \Rb^{n_\xi \times 1}$.  \;
  Compute sparse $LU$-factorisation of $A$, 
  \[
    LU=A \in \Rb^{n_\xi \times n_\xi}.
  \]\;
  \For{$n = 1,\ldots,n_t$}{
    Set $V = V^{n-1} \in \Rb^{n_\xi \times n_x} $. \;
    Compute the external input at time $t_{n-1}$ and store it in $G \in \Rb^{n_\xi \times n_x}$.\;
    Set $z = \cF_{n_x}\big[(\alpha' \odot \sigma)^T S(V)\big] \in \Rb^{1 \times n_x} $.\;
    Set $N =  h_x \alpha \cF^{-1}_{n_x} [ \hat w \odot  z] \in \Rb^{n_\xi \times n_x}$.\;
    Solve for $V^{n}$ the linear problem $(LU)V^n = V + \tau(N+G)$.
  }
}
\end{algorithm}

\begin{algorithm}
\caption{IMEX time stepper in vector form \eqref{eq:ODEs}, nonlinear term evaluated with
quadrature formula \eqref{eq:NQuad}.}
\label{alg:naive}
\SetAlgoNoLine
\DontPrintSemicolon
\SetKwInOut{Input}{Input}
\SetKwInOut{Output}{Output}

\Input{Initial condition $U^0 \in \Rb^{n_\xi n_x}$, time step $\tau$, number of steps $n_t$.}
\Output{An approximate solution $(U^n)_{n=1}^{n_t} \subset \Rb^{n_\xi n_x}$}
\Begin{
\BlankLine
  Compute grid vectors $\xi \in \Rb^{n_\xi}$, $x \in \Rb^{n_x}$. \;
  Compute synaptic vector $w \in \Rb^{1 \times n_x}$.\;
  Compute synaptic vectors $\alpha, \alpha' \in \Rb^{n_\xi \times 1}$. \;
  Compute quadrature weights $\rho \in \Rb^{n_x}$, $\sigma \in \Rb^{n_\xi}$.  \;
  Compute sparse $LU$-factorisation 
  \[
    LU = \big( 
      (1+\tau \gamma) I_{n_x n_\xi} - \tau \nu I_{n_x} \otimes D_{\xi \xi} 
	\big)
      \in \Rb^{n_\xi n_x \times n_\xi n_x }.
  \]
\BlankLine
  \For{$n = 1,\ldots,n_t$}{
    Set $Z = U^{n-1} \in \Rb^{n_\xi n_x} $. \;
    Compute the external input at time $t_{n-1}$ and store it in $G \in \Rb^{n_\xi n_x}$. \;
    Compute the nonlinear term $N$ using \eqref{eq:NQuad}.\;
    Solve for $U^{n}$ the linear problem $(LU)U^n = Z + \tau(N+G)$.
  }
}
\end{algorithm}

In Table~\ref{tab:flops} we count flops required in each line of Algorithms
\ref{alg:smart} and \ref{alg:naive}. The data is grouped so as to distinguish between
the initialisation phase of the algorithms, and the iterations for the time steps.
Algorithm \ref{alg:smart} outperforms substantiatlly Algorithm \ref{alg:naive} in
both phases. In the initialisation, the number of flops scales linearly for Algorithm
\ref{alg:smart}, and quadratically for Algorithm \ref{alg:naive}. This is mostly
owing to the $LU$-factorisation step, which involves the $n_\xi$-by-$n_\xi$ matrix
$A$ in the former, and an $n_\xi n_x$-by-$n_\xi n_x$ matrix in the latter. 

The efficiency gain is more striking in the cost per time step: owing to the
pseudospectral evaluation of the nonlinearity, only $O(n_\xi n_x) + O(n_x \log n_x)$
flops are necessary in Algorithm \ref{alg:smart}, as opposed to $O(n^2_\xi n^2_x)$
flops in Algorithm \ref{alg:naive}.

\begin{table}
  \centering
  \caption{Flop count for the initialisation step (lines 2--6) and for one time
  step (lines 8--12) in Algorithms 1,2.}
  \label{tab:flops}
\begin{tabular}{cccc}
  \toprule
  \multicolumn{2}{c}{Algorithm 1} & \multicolumn{2}{c}{Algorithm 2} \\
  Lines & Flops & Lines & Flops \\
  \otoprule
   2    & $  n_\xi + n_x        $  &  2   & $  n_\xi + n_x                  $ \\
   3    & $  2n_x               $  &  3   & $     n_x                       $ \\
   4    & $  2n_\xi             $  &  4   & $    2n_\xi                     $ \\
   5    & $  n_\xi              $  &  5   & $   n_\xi + n_x                 $ \\
   6    & $  2n_\xi-1           $  &  6   & $   2n_\xi n_x -1               $ \\
  \midrule
  2--6 & $  O(n_\xi) + O(n_x)  $  & 2--6 & $ O(n_\xi n_x) $ \\
  \midrule
   8    & $  n_\xi n_x                                              $  &  8   & $  n_\xi n_x                            $ \\
   9    & $  n_\xi n_x                                              $  &  9   & $  n_\xi n_x                            $ \\  
   10   & $  3 n_\xi n_x + O(n_x \log n_x) + n_\xi - n_x               $  &  10  & $  2n^2_\xi n_x^2 - n^2_\xi n_x         $ \\
   11   & $  n_\xi n_x + O(n_x \log n_x) + 2n_x                        $  &  11  & $  5n_\xi n_x -4                        $ \\
   12   & $  5 n_\xi n_x - 4 n_x                                    $  &      & $                                       $ \\
  \midrule
   8--12 & $O(n_\xi n_x) + O(n_x \log n_x)$  & 8--11 & $ O(n^2_\xi n^2_x) $ \\
  \bottomrule
%
\end{tabular}
\end{table}

\begin{table}
  \centering
  \caption{Space requirements, measured in Floating Point Numbers, for
  Algorithms 1 and 2. Arrays $d_1, \ldots, d_3$, store diagonals of the
$LU$-factorisation in the respective algorithms.}
\label{tab:memory}
\begin{tabular}{ccc}
  \toprule
   Floating Point Numbers & Algorithm 1 & Algorithm 2 \\
  \otoprule
  $n_\xi$       & $\xi,\alpha,\alpha',d_1,d_2,d_3,z$ & $\xi, \rho, \alpha, \alpha'$ \\  
  $n_x$         & $x,w,\sigma$                       & $x, \sigma, w$ \\  
  $n_\xi n_x$   &  $V,V^n,G,N$                       & $U^n,Z,N,G,d_1,d_2,d_3$ \\  
  \midrule
  Total         & $4n_x n_\xi + 7n_\xi + 3n_x$      &  $ 7 n_x n_\xi +2n_\xi + 2n_x $  \\  
  \bottomrule
\end{tabular}
\end{table}

An important point to note that, in the case of a 2D somatic space with, say
coordinates $(x,y,\xi)$ and $n_x=n_y=N$, $n_\xi$ grid points (see
Figure~\ref{fig:sketch}(a)), the size of the matrix $A$ in
Algorithm \ref{alg:smart} remains unaltered, while Algorithm \ref{alg:naive}
requires the factorisation and inversion of a much larger matrix, of size $n_\xi N^2$-by-$n_\xi
N^2$. Estimates for the efficiencies in this case can be obtained by replacing $n_x$ by
$N^2$ in the table, leading to much greater savings.

Finally, in Table~\ref{tab:memory} we collect the variables used by both algorithms,
and count the storage requirement of each of them, measured floating point numbers.
The results show that Algorithm 1 requires asymptotically the same storage as
Algorithm 2 $O(n_\xi n_x)$. For fixed values of $n_\xi$ and $n_x$, however, the
latter uses almost twice as much storage space as the former.

\section{Travelling waves}\label{sec:TWTest}
\begin{figure}
  \centering
  \includegraphics{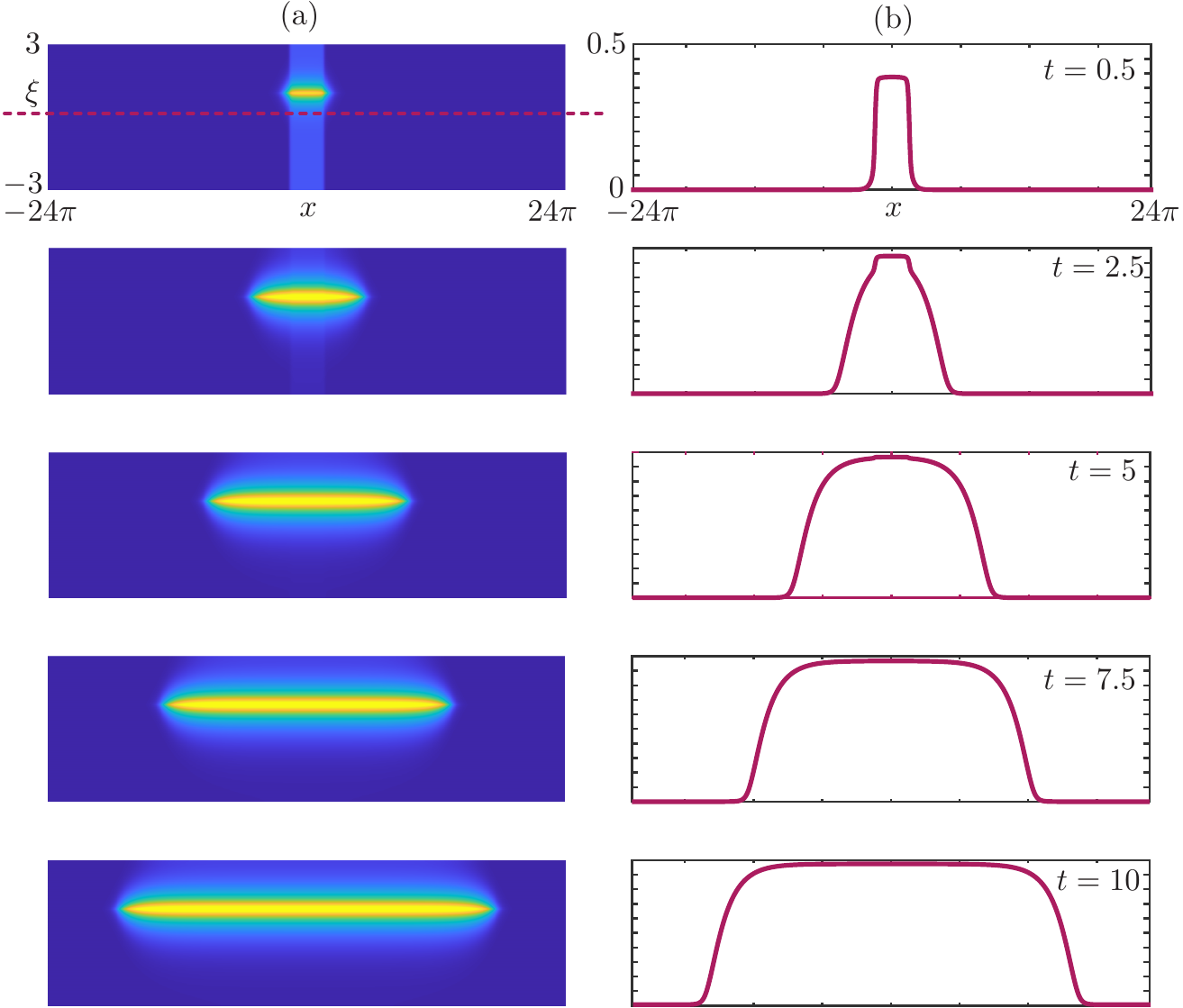}
  \caption{ Coherent structure observed in time
    simulation of \eqref{eq:systemNum}, \eqref{eq:TWSigmoidal}, \eqref{eq:TWKernel}.
    (a): Pseudocolor plot of $V(x,\xi,t)$ at several time points, showing two
    counter-propagating waves. (b): Solution at $\xi = 0$, showing the wave
    profile. Parameters: $\xi_0 = 1$, $\ep = 0.005$, $\nu=0.4$, $\gamma=1$, $\beta=1000$,
    $\theta =0.01$, $\kappa =3$, $L_x = 24 \pi$, $L_\xi =3$, $n_x =
    2^{10}$, $n_\xi=2^{12}$, $\tau = 0.05$.}
    \label{fig:wave}
\end{figure}
We tested the algorithm on an analytically tractable neural field problem, and we
report in this section our numerical experiments. For the test, we take a sigmoidal
firing rate function
\begin{equation}\label{eq:TWSigmoidal}
  S(V) = \frac{1}{1 + \exp(-\beta(V-\theta))}, 
\end{equation}
and kernel specified by
\begin{equation}\label{eq:TWKernel}
  w(x) = \frac{\kappa}{2} \exp \bigg( -\frac{|x|}{2} \bigg),
  \qquad
  \delta_\ep(\xi) = \frac{1}{\ep
  \sqrt{\pi}} \exp\bigg(-\frac{\xi^2}{\ep^2}\bigg) 
\end{equation}
where $\beta, \theta, \kappa$ are positive constants. If $S(V) = H(V- \theta)$, $H$
being the Heaviside
function, $\delta_\ep$ is replaced by the Dirac delta distribution, and the
evolution equation is posed on $\Rb^2$, then the model supports solutions for which
$V(x,0,t)$ is a travelling front 
$
V(x,0,t) = \varphi(x-v_*t)$, 
with
$\varphi(\pm \infty) = V_\pm$,
whose speed $v_*$ satisfies the implicit equation~\cite{Ross2019}
\begin{equation}\label{eq:analyticalSpeed}
  \frac{\kappa \exp(-\psi(v_*,\nu) \xi_0 )}{2 \psi(v_*,\nu) \nu} -\theta = 0,
  \qquad \psi(v_*,\nu) = \sqrt{\frac{\gamma + v_*}{\nu}}.
\end{equation}
\begin{figure}
  \centering
  \includegraphics{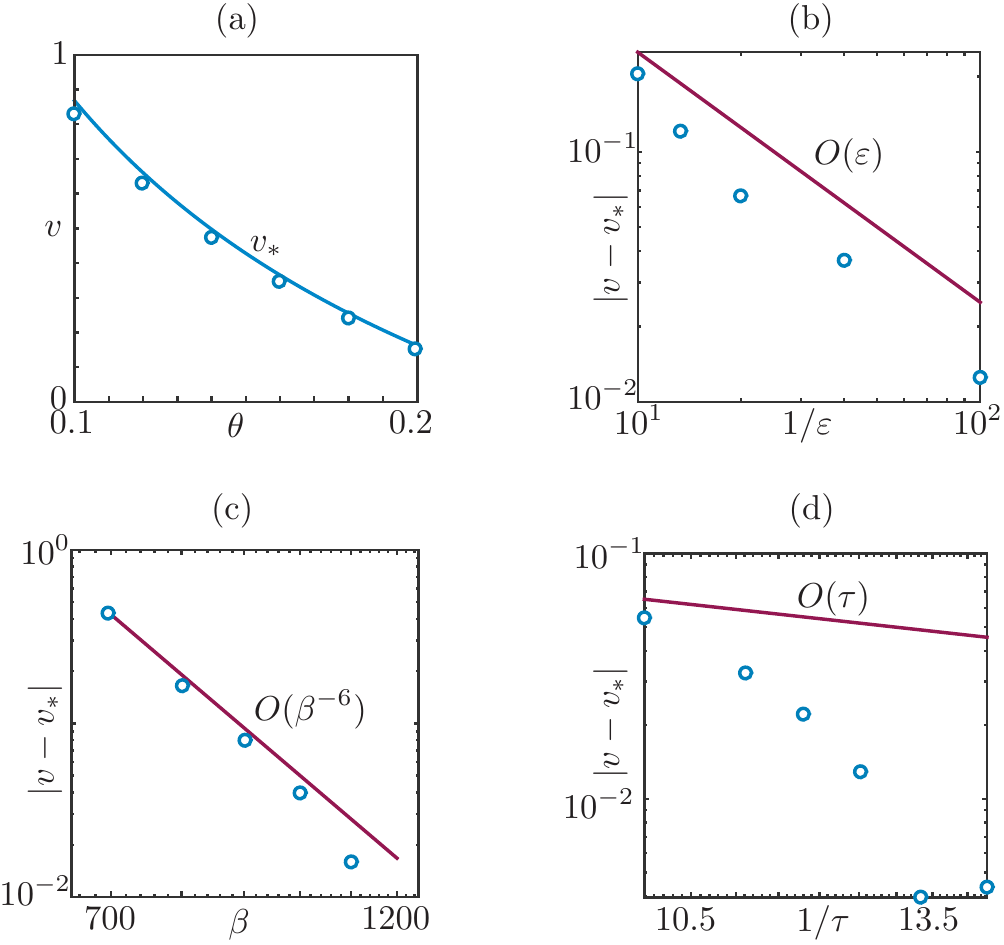}
  \caption{(a) Travelling wave speed versus firing rate
    threshold, computed analytically via \eqref{eq:analyticalSpeed}, and numerically
    via the time-stepper. (b)--(d) Convergence of the computed speed to the
    analytical speed at $\theta = 0.01$, as a function of the the kernel support
    parameter $\ep$, the steepness of the sigmoid $\beta$, and the time-stepping
    parameter $\tau$, respectively. Parameters as in \eqref{fig:wave}.}
    \label{fig:convergenceTests}
\end{figure}

To test our scheme we study solutions to \eqref{eq:systemNum},
\eqref{eq:TWSigmoidal}, \eqref{eq:TWKernel} with
$L_x, \beta \gg 1$, $L_\xi \gg \sqrt{\nu /\gamma}$,  (the
characteristic electrotonic length), and $\ep \ll 1$.
Since for this problem $[-L_x,L_x)
\cong \mathbb{S}$, we expect to observe at $\xi=\xi_0$
two counter-propagating waves with approximate speed $v$ and 
\[
  V(\pm L_x,\xi_0,t) \approx V_+, \qquad V(0,\xi_0,t) \approx V_-.
\]
We show an exemplary profile of this coherent structure in Figure~\ref{fig:wave},
where we observe two counter-propagating waves, as described above. 

Since the wavespeed $v_*$ is available implicitly, we performed some tests
to validate the proposed algorithm. Firstly, we compute roots of
\eqref{eq:analyticalSpeed} in the variable $v_*$, as a function of the firing rate
threshold $\theta$. In Figure~\ref{fig:convergenceTests}(a) we observe a good agreement with
the wavespeed observed in direct simulations. The latter has been
computed by post-processing data from numerical simulations: using first-order
interpolants we approximate a positive function $x_*(t)$ such that $V(x_*(t),0,t) =
\theta$, that is, the $\theta$-level set of $V(x,0,t)$ on $[0,L_x] \times [0,T]$;
after an initial transient, $\dot x_*(t)$ is approximately constant and provides an
estimate of $v_*$, which is derived via first-order finite differences. In
Figure~\ref{fig:convergenceTests}(a) we observe a small discrepancy, which should be expected as
we have several sources of error, namely: the time-stepping error, the spatial
discretisation error for the differential and integral operators, the error due to
the sigmoidal firing rate and to $\delta_\ep$ (the theory is
valid for Heaviside firing rate and for a Dirac-delta distribution $\delta$). In
Figures~\ref{fig:convergenceTests}(b)--(d) we show convergence plots for these errors (except for
the second-order spatial discretisation error which is dominated in numerical
simulations by the first-order time-stepping error).

\section{Turing instability}\label{sec:TuringTest}
The model defined by (\ref{1}), with an appropriate choice of somatic interaction,
can also support a Turing instability to spatially periodic patterns
\cite{Bressloff96}.  These in turn may either be independent of time or periodic in
time.  In the latter case this leads to periodic travelling waves.  Whether emergent
patterns be \textit{static} or \textit{dynamic} they both provide another validation
test for the numerical scheme presented here, as the bifurcation point as determined
analytically from a Turing analysis should agree with the onset of patterning in a
direct numerical simulation.  A relatively recent description of the method for
determining the Turing instability in a neural field with dendritic processing can be
found in \cite{Coombes2014}.  Here we briefly summarise the necessary steps to arrive
at a formula for the continuous spectrum, from which the Turing instability can be
obtained.

In general a homogeneous steady state solution of (\ref{1}) will only exist if either
$S(0)=0$ or $\int_{\Rb^2} W(x,\xi,y,\eta)\diff y \diff \eta = \text{constant}$
for all $(x,\xi)$.  The latter condition is not generic, and so for the purposes of
this validation exercise we shall work with the choice $S(0)=0$ for which $V=0$ is
the only homogeneous steady state.  Linearising around $V=0$ and using
(\ref{eq:kernel}) gives an evolution equation for the perturbations $\delta
V(x,\xi,t)$ that can be written in the form
\begin{equation}
\delta V(x,\xi,t) = S'(0) \int_{-\infty}^t \Theta(\xi-\xi_0,t-s) \int_{\Rb} 
w(|x-x'|) \delta V(x',0,s) \diff x' \diff s,
\label{deltaV}
\end{equation}
where
\[
  \Theta(\xi, t)=\mathrm{e}^{-\gamma t} \frac{\mathrm{e}^{-\xi^{2} /(4 \nu
  t)}}{\sqrt{4 \pi \nu t}} H(t).
\]
Focusing on a \textit{somatic} field $\delta V(x,0,t)$, we see from (\ref{deltaV})
(with $\xi=0$) that this has solutions of the form $\e^{\lambda t} \e^{i p x}$ for
$\lambda \in \Cb$ and $p \in \Rb$, where $\lambda=\lambda(p)$ is defined by the
implicit solution of $\mathcal{E}(\lambda,p)=0$, where
\begin{equation}
\mathcal{E}(\lambda,p) = 1 - S'(0) \frac{\exp(-\psi(\lambda,\nu) \xi_0 )}{2 \psi(\lambda,\nu) \nu} \widehat{w}(p) .
\end{equation}
Here the function $\psi$ is defined as in (\ref{eq:analyticalSpeed}) and
$\widehat{w}(p)$ is the Fourier transform of $w$:
\begin{equation}
\widehat{w}(p) = \int_{\Rb}  w(|y|) \e^{-i p y} \diff y.
\end{equation}
We note that since $w(x)=w(|x|)$ then $\widehat{w}(p) \in \Rb$.  

\begin{figure}
  \centering
  \includegraphics{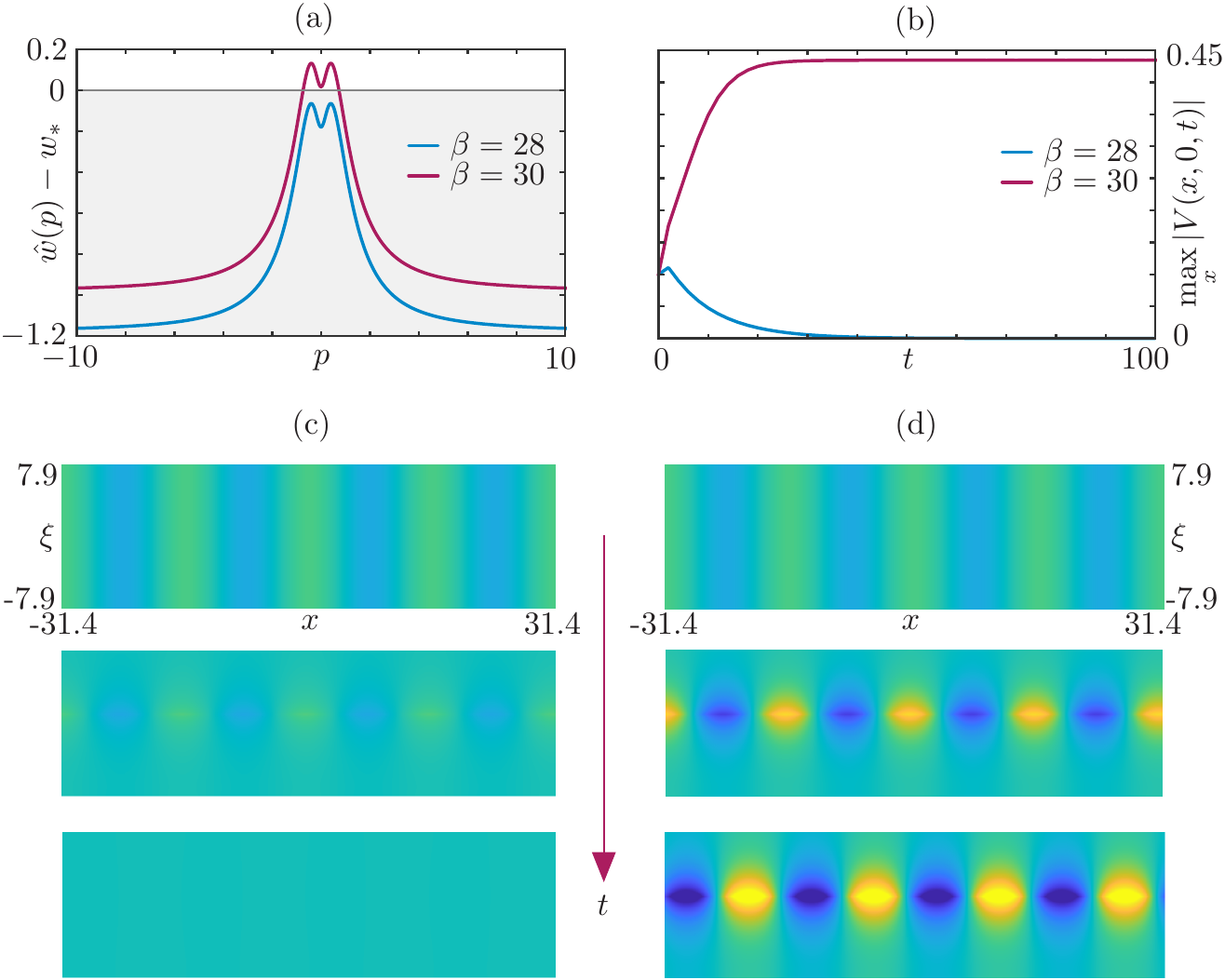}
  \caption{Numerical simulation of Turing bifurcation for the model with kernel and
   firing rate function given in \eqref{eq:mexicanHatKernel}. (a): Plots of $\hat w(p) - w_*$
   for $\beta = 28$ and $\beta=30$, from which we deduce
that a Turing bifurcation occurs for an intermediate value of $\beta$. We expect
perturbations of the trivial state to decay exponentially if $\beta=28$ and to increase
exponentially if $\beta = 30$, as confirmed in panels (b)--(d). (b): Maximum absolute
value of the voltage at $\xi=0$, as function of time, when the trivial steady state is
perturbed. (c),(d): Pseudocolor plot of $V$ when $\beta=28$ and $\beta=30$, respectively.
Parameters: $\xi_0 = 1$, $\ep = 0.005$, $\nu  = 6$, $c = 1$, $a_1 = 1$, $b_1 =
1$, $a_2 = 1/4$, $b_2 = 1/2$, $n_x  = 2^9$, $L_x  = 10\pi$, $n_\xi = 2^{11}$, $L_\xi =
2.5 \pi$, $\tau = 0.01$. }
  \label{fig:TuringStatic}
\end{figure}

Note that if $\lambda \in \Rb$ with $\lambda > -\gamma$ then $\mathcal{E}(\lambda,p) \in
\Rb$.
The trivial steady state is stable to spatially-periodic perturbations if 
\[
  \hat w(p) < w_* = 2\psi(0,\nu) \nu \exp(\psi(0,\nu) \xi_0)/S'(0)  \qquad
  \text{for all $p \in \Rb$}.
\]
Hence a static instability occurs under a parameter variation when $w(p_*) = w_*$ for
some $p_* \in \Rb$, the critical wavelength of the unstable pattern.
Hence if $\widehat{w}(p)$ has a positive peak away from the origin, at $p=p_*$, then a static
Turing instability can occur (see Figure \ref{fig:TuringStatic}(a)). This is possible
if $w(|x|)$ has a Mexican-hat shape, describing short range excitation and long range
inhibition. 

We have validated this scenario numerically, by simulating a neural field with
\begin{equation}\label{eq:mexicanHatKernel}
  w(x) = a_1 \exp(-b_1 |x|) - a_2 \exp(-b_2 |x|), \qquad S(V) = \frac{1}{1 +
  \exp{(-\beta V)}} - \frac{1}{2},
\end{equation}
and reporting results in Figure~\ref{fig:TuringStatic}. In the Figure we pick the
steepness of the sigmoidal firing rate as main parameter, deduce that
a Turing bifurcation occurs for a critical value $\beta_* \in [28,30]$, perturb the
trivial steady state by setting initial condition $V_0(x,\xi) = 0.01 \cos(p_* x)$ and
domain $\bar \Omega = [-4\pi/p_*, 4\pi/p_*] \times [-\pi/p_*,\pi/p_*]$, and observe
the perturbations decaying for $\beta = 28$, and growing for $\beta = 30$, respectively.

Note that, if $\lambda \in \Cb$, then $\mathcal{E}(\lambda,p) \in \Cb$ and it is
possible that a dynamic Turing instability can occur, with an emergent frequency $\omega_c$.  It is
known that this case is more likely for an inverted Mexican-hat shape, describing
short range inhibition and long range excitation \cite{Bressloff96}. We do not show
computations for this case, but we briefly discuss it below. The dynamic
bifurcation condition can be defined by tracking the continuous spectrum at the point
where it first crosses the imaginary axis away from the real line.  This is
equivalent to solving $P_p H_\omega - H_p P_\omega = 0$ with $P(0,\omega_c,p) = 0 =
H(0,\omega_c,p)$, where the subscripts denote partial differentiation and
$P(\nu,\omega,p) = \Real \, \mathcal{E}(\nu + i \omega,p)$ and $H(\nu,\omega,p) =
\Imag \, \mathcal{E}(\nu + i \omega,p)$ \cite[Chapter 1]{Coombes2005}.

\section{Conclusions} \label{sec:conclusions}
In this paper we have presented an efficient scheme for the
numerical solution of neural fields that incorporate dendritic processing. The model
prescribes diffusivity along the dendritic coordinate, but not along the cortex; in
addition, the nonlinear coupling is nonlocal on the cortex, but essentially local on
the dendrites. This structure allows the formulation of a compact numerical scheme,
and provides efficiency savings both in terms of operation counts, and in terms of
the space required by the algorithm. Firstly, a small diffusivity differentiation
matrix is decomposed at the beginning of the computation, and then used repeatedly
to solve a set of linear problems in the cortical direction. This aspect of the
computation is appealing, especially for high-dimensional computations where a 2D
cortex is coupled to the 1D dendritic coordinate, as the decomposition is performed
once, and involves only a 1D differentiation matrix. Secondly, the largest
computational effort of the scheme, which is in the evaluation of the nonlinear term, can be
reduced considerably using DFTs. We
have also provided a basic numerical analysis of the scheme, under the assumption that a
solution to the infinite-dimensional problem exists. The existence of this solution
remains an open problem, which will be addressed in future publications. 

The numerical implementation presented here does not exploit the fact that the
synaptic kernel is localised via the function $\delta_\ep$. If one models the kernel
using a compactly supported function, for instance
\begin{equation}\label{eq:compactDelta}
  \delta_\ep(\xi) = \kappa
  \exp\bigg(-\frac{\xi^2}{\ep^2}\bigg) 1_{(-\ep,\ep)}(\xi), 
\end{equation}
which is supported in a small interval of $O(\ep)$ length, its
evaluation at the grid points is nonzero only on a small index set, namely
\[
  \delta_\ep(\xi_i-\xi_0) =
  \begin{cases}
    \alpha_i & \text{if $i \in \Ib$,} \\
    0        & \text{otherwise,}
  \end{cases}
  \qquad
  \delta_\ep(\xi_{i'}) =
  \begin{cases}
    \alpha'_{i'} & \text{if $i' \in \Ib'$,} \\
    0        & \text{otherwise,} 
  \end{cases}
\]
where $\Ib,\Ib' \subseteq \Nb_{n_\xi}$ are index sets with $O(\ep/L_{n_\xi})$
elements $|\Ib|, |\Ib'| \ll n_\xi$, respectively. This implies
\[
 N_{ij}(V) = \alpha_i \sum_{j' \in \Nb_{n_\xi}} w_{j-j'} \rho_{j'}
  \sum_{i' \in \Ib'} \alpha'_{i'} \sigma_{i'} S(V_{i'j'})
  \qquad
  (i,j) \in \Ib \times \Nb_{n_x},
\]
and we note that only $|\Ib|$ rows of $N$ are nonzero, and the inner sum is only over
$|\Ib'|$ elements. The formula above evaluates the nonlinear term $N$ in just
$(2|\Ib| + |\Ib'|)n_x + O(n_x \log n_x) = O(n_x + n_x\log n_x)$ operations. Numerical
experiments and convergence tests have been performed also for this formula, albeit
the results not presented here, because a synaptic kernel with the choice
\eqref{eq:compactDelta} is no longer in $C^2(\Omega \times \Omega)$, hence
Lemma~\ref{lem:IMEXconvergence} does not hold, and we plan to provide a convergence
result for this case elsewhere. We provide, however, a Matlab implementation of this
code in Appendix~\ref{sec:matlab}.

Possible extensions of the model include curved geometries \cite{Visser:2017hy,Bojak2010}, which should
benefit from a similar strategy used here for the dendritic coordinate, as well as
multiple population models. We expect that the latter will induce different coherent
structures to the ones reported here. The method outlined in this paper is valid also
in the context of numerical bifurcation analysis, which can be employed to study the
bifurcation structure of steady states and travelling waves.

The inclusion of synaptic processing to the present model is straightforward, by
coupling \eqref{eq:systemNum} to an equation of type $Q\Psi =K$ where $Q=
(1+\alpha^{-1} \partial_t)^2$ is a temporal differential operator. The resulting
model would not involve any additional spatial differential or integral
operator, therefore the proposed scheme can be applied by simply augmenting the
discretised state variables.
%

It is also important to address the role of axonal delays on the generation of large
scale brain rhythms.  A recent paper \cite{Ross2019} has explored how this might be
done in a purely PDE setting, generalising the Nunez brain-wave equation to include
dendrites. A natural extension of the work in this paper is to consider a more
general numerical treatment of a model with both dendritic processing and
space-dependent axonal delays in an integro-differential framework.

\section*{Acknowledgments} P.M. Lima acknowledges support from Funda\c c\~ao para a
Ci\^encia e a Tecnologia (the Portuguese Foundation for Science and Technology)
through the  grants POCI-01-0145-FEDER-031393 and UIDB/04621/2020.


\newpage
\appendix
\section{Matlab implementation}\label{sec:matlab}
\phantom{In this section we provide a listing of the code used}
\begin{lstlisting}
%% Model parameters
xi0 = 1.0; epsi = 0.005; nu  = 0.4; xStar = 5; c = 1;
mu  = 1000; vth = 0.01; kappa = 3;

%% Numerical parameters
nx  = 2^12; Lx  = 24*pi; nxi = 2^10; Lxi = 3;
tau = 0.05; nt = 1000; iplot = 10; ihist = iplot;

%% Spatial grid
hx  = 2*Lx/nx;  x  = -Lx +[0:nx-1]*hx;
hxi = 2*Lxi/(nxi-1); xi = -Lxi + [0:nxi-1]'*hxi;
[X,XI] = meshgrid(x,xi);

%% Function handles
wFun      = @(x) kappa*0.5*exp(-abs(x)); 
alphaFun  = @(xi,epsi) (epsi*sqrt(pi))^-1 * exp( -(xi-xi0).^2/epsi^2)...
                 .*(abs(xi - xi0) <= 2*epsi);
alphaPFun  = @(xi,epsi) (epsi*sqrt(pi))^-1 * exp( -xi.^2/epsi^2)...  
                 .*(abs(xi) <= 2*epsi);
SFun      = @(v) 1./(1+exp(-mu*(v-vth))); 

%% Precomputing vectors
wHat = fft(wFun(x));
alpha  = alphaFun(xi,epsi);  J  = find(alpha  ~= 0); alpha  = alpha(J);
alphaP = alphaPFun(xi,epsi); JP = find(alphaP ~= 0); alphaP = alphaP(JP);

%% Quadrature weights
sigma = ones(nxi,1); sigma([1 nxi]) = 0.5; sigma = hxi*sigma; sigma = sigma(JP);

%% Differenstiation matrix and linear operator
e = ones(nx,1); D2 = spdiags([e -2*e e], -1:1, nxi, nxi); 
D2(1,2) = 2; D2(nxi,nxi-1) = 2; D2 = D2/hxi^2;
A  = (1+tau/c)*speye(nxi) - tau*nu*D2;
dA = decomposition(A);

%% Initial condition
t = 0;
V = 0.5*(1 - 1./(1+exp(-5*(X-xStar)))).*(X>0) +...
    0.5*(1 - 1./(1+exp(5*(X+xStar)))).*(X<=0);

%% Plots
subplot(2,2,[1 2]); 
p1 = surf(X,XI,V); p1.ZDataSource = 'V'; 
shading interp; view([0 90]); axis tight; caxis([0 1]); 

subplot(2,2,[3 4]); 
[~,id0] = min(abs(xi-xi0)); VSlice = V(id0,:); 
p2 = plot(x,VSlice,'*-'); p2.YDataSource = 'VSlice'; ylim([0 2.5]); 

%% TimeStep 
for it = 1:nt

  % Compute F
  F = zeros(nxi,nx);
  S = (alphaP .* sigma)'*SFun(V(JP,:));
  convol = hx*ifftshift(real(ifft(fft(S) .* wHat)));
  F(J,:) = alpha * convol;

  % Update V
  V = dA\(V + tau*F);
  t = t + tau;

  % Update plot
  if mod(it,iplot) == 0  
    VSlice = V(id0,:);
    title(['t = ' num2str(t)]); caxis([0 2]); 
    refreshdata, drawnow;
  end

end
\end{lstlisting}

\bibliographystyle{siamplain}
\bibliography{references}
\end{document}